\documentclass[12pt]{amsart}

\usepackage{amsmath, amscd, amssymb}
\usepackage{latexsym, amssymb, amsmath, amscd, amsfonts}
\usepackage[mathscr]{eucal}
\usepackage{mathrsfs}
\usepackage{cancel}
\usepackage{concrete,euler,textcomp}
\usepackage{youngtab}
\usepackage[T1]{fontenc}

\def\lab{\label}

\numberwithin{equation}{section}

\newcommand{\Gr}{\operatorname{Gr}}

\newcommand{\CC}{\mathbb{C}}

\newcommand{\LL}{\mathbb{L}}
\newcommand{\N}{\mathbb{N}}

\newcommand{\ZZ}{\mathbb{Z}}

\newcommand{\bn}{\mathbf{n}}

\newcommand{\bt}{\mathbf{t}}

\def\bd{{\bf d}}
\def\mat{{\rm Mat}}
\def\Rep{{\rm Rep}}

\def\sO{{\mathscr O}}

\def\2M{M}

\newcommand{\GL}{{\rm GL}}

\def\lra{\longrightarrow}

\newcommand{\Ga}{\Gamma}

\newcommand{\ee}{{\bf e}}

\def\begeq{\begin{equation}}
\def\endeq{\end{equation}}
\def\and{\quad{\rm and}\quad}

\def\br{\bigr)}

\def\and{\quad\text{and}\quad}

  \DeclareMathOperator{\Hom}{Hom}

\let\lab=\label

\newtheorem{prop}{Proposition}[section]
\newtheorem{theo}[prop]{Theorem}
\newtheorem{lemm}[prop]{Lemma}
\newtheorem{coro}[prop]{Corollary}
\newtheorem{defi}[prop]{Definition}

\theoremstyle{definition}
\newtheorem{say}[prop]{}
\newtheorem{rema}[prop]{Remark}

\def\bt{{\bf t}}

\let\lab=\label

\def\lab#1{\label{#1}[{#1}]\  }

\def\lab{\label} 

\def\beq{\begin{equation}}
\def\eeq{\end{equation}}

\def\br{{\bf r}}
\def\bn{{\bf n}}

\def\bh{{\bf h}}
\def\bt{{\bf t}}

\begin{document}

\title[Quivers, Invariants and GM Correspondence]{Quivers, Invariants and Quotient
Correspondence}

\author{Yi Hu}
\address{ Department of Mathematics\\
University of Arizona, Tucson, USA}
\email{yhu@math.arizona.edu}

\author{Sangjib Kim}
\address{Department of Mathematics\\
Ewha Womans University, Seoul, South Korea}
\email{sangjib@gmail.com}

\thanks{The first author was partially supported by NSA grant MSP07G-112  and NSF DMS 0901136.}

\begin{abstract}
This paper studies the geometric and algebraic aspects of the moduli spaces of quivers of fence type.
We first provide two  quotient presentations of the quiver varieties and interpret their equivalence
as a generalized  Gelfand-MacPherson correspondence. Next, we introduce parabolic quivers and extend 
the above from the actions of reductive groups to the actions of parabolic subgroups.  Interestingly, 
the above geometry  finds its  natural counterparts  in the representation theory as the branching rules 
and transfer principle in the context of  the reciprocity algebra.  The last half of the paper 
establishes this connection.
\end{abstract}

\maketitle

\bigskip
\section{Introduction}
\medskip

The usual Gelfand-MacPherson correspondence (\cite{GM82}),
as formulated by Kapranov (\cite{Ka93}),  establishes a natural correspondence
between GIT quotients of  $\left( \mathbb{P}^{n-1}\right)^{k}$  by the diagonal action 
of $\GL_{n}(\mathbb{C)}$ and GIT quotients of the Grassmannian variety $\Gr(n,\mathbb{C}^{k})$ 
by the maximal torus $(\mathbb{C}^{\ast })^{k}$.

In this paper, we generalize the above by providing two versions of  the  quotient correspondence for  moduli spaces of
quivers of fence type (\S\S \ref{reductive}, \ref{parabolic}). 
A quiver is an oriented graph $Q=(Q_0, Q_1,  \bh, \bt)$ equipped with a finite ordered set of vertices 
$Q_0$, a set of arrows $Q_1$, and two functions $\bh, \bt$ such that for each arrow 
$a \in Q_1$, $\bh(a) \in Q_0$ is the head and  $\bt(a) \in Q_0$ is the tail. 
It is of fence type if the vertex set $Q_0$ can be decomposed as the disjoint union 
of subsets $H$ and $T$ such that $H$ consists of only heads of arrows and $T$ consists of only tails.
Associated to such a quiver  are products of  general linear groups
$$G_H= \prod_{h \in H} \GL_{d_h} \and G_T= \prod_{t \in T} \GL_{d_t} $$
where  $\bd=(d_q)_{q \in Q_0}$ is a fixed  dimension vector. If for any $h \in H$ and $t \in T$, we set
$$n_h = \sum_{a \in Q_1, \bh(a)=h} d_{\bt(a)} \and n_t=\sum_{a \in Q_1, \bt(a)=t} d_{\bh(a)},$$
then, we can associate to the quiver the following  products of Grassmannians
$$X_T = \prod_{t \in T} \Gr(d_t, \CC^{n_t}) \and X_H= \prod_{h \in H} \Gr(d_h, \CC^{n_h}).$$
The group $G_H$ acts on $X_T$ naturally; the group $G_T$ acts on $X_H$ naturally.
For any ${\bf e}=(e_t)_{t \in T} \in \N^{|T|}$ and  ${\bf r}=(r_h)_{h \in H} \in \N^{|H|}$, we have an ample line bundle
$$L_\ee =\boxtimes_{t \in T} \sO_{ \Gr(d_t, \CC^{n_t})} (e_t)$$
over $X_T$ determined by $\ee$.
The tuple ${\bf r}$ defines a character  of $G_H$: $$\chi_\br: G_H \to \CC^*$$  (see \eqref {characters}) and thus induces
 a $G_H$-linearization $L_\ee(\br)$ over $X_T$.  Similarly but with the roles of $\br$ and $\ee$ swapped,
we have an ample line bundle over $X_H$ $$L_\br =\boxtimes_{h \in H} \sO_{ \Gr(d_h, \CC^{n_h})} (r_h),$$
a $G_T$-character $\chi_\ee: G_T \to \CC^*$ and the induced
 $G_T$-linearization $L_\br(\ee)$ over $X_H$.

\begin{theo}\lab{reductiveGM2} 
There is a natural one-to-one correspondence between the set of GIT quotients of $X_T$ by $G_H$
and the set of GIT quotients of $X_H$ by $G_T$. Precisely,
suppose that $\br \in \N^{|H|}$ and $\ee \in \N^{|T|}$ satisfy the compatibility condition \eqref{rd=ed},
then we have a natural isomorphism between $X_T^{ss}(L_\ee(\br))/\!/G_H$ and
$X_H^{ss}(L_\br(\ee))/\!/G_T$. 
\end{theo}

As a special case, when the quiver is a star quiver, that is, it has a unique head ($H$ consists of a single element),
we recover the quotient correspondence of \cite{Hu05}  (of which the usual GM correspondence is a special case).

\medskip

The above are quotient correspondences for reductive group actions. In  some practice, one may encounter
quotients by parabolic groups which often requires special treatments as there is no general quotient theory for
non-reductive groups. In \S \ref{parabolic}, we consider the parabolic subgroup actions on the representation space
of the quiver.  It turns out their quotients  parameterize what we call $``$parabolic quivers$"$:
a parabolic quiver  is a representation of the quiver $Q$ together with some (partial) flags of $V_b$
at every vertex $b \in Q_0$.    To specify the flags,  for any vertex $v \in Q_0$,
  we fix a partition $$d_v= d_{v_1} + \cdots +d_{v_s}$$ of $d_v$ by positive integers
  where $s=s(v)$ is a positive integer depending on the vertex $v$.
Associated to each $v \in Q_0$, we have  the following (partial) flag variety
\begin{equation*}
Y_v:=\left\{ \left( 0\subset V_{1}\subset
\cdots \subset V_{s}\subset \mathbb{C}^{n_v}\right) :\dim
V_{i}=\sum_{j=1}^{i}d_{v_j}\right\}
\end{equation*}
where $n_v$ is as define earlier (see also \eqref{nv}).  We set
$$Y_T= \prod_{t \in T} Y_t \and Y_H = \prod_{h \in H} Y_h.$$
We also let
$$P_H= \prod_{h \in H} P_h \and P_T = \prod_{t \in T} P_t$$
where   $P_v$ as the parabolic subgroup of $\GL_{d_v}$  as defined in \eqref{Parabolic}.
Then $P_H$ acts on $Y_T$ naturally and $P_T$ acts on $Y_H$ naturally.

\begin{theo}[{\protect = Theorem \ref{parabolicGM}}]
There is an one-to-one correspondence between the set of GIT quotients of
$Y_T$ by $P_H$ and the set of GIT quotients of $Y_H$ by $P_T$. 
\end{theo}

This extends the quotient correspondence from the general linear groups to 
parabolic subgroups. For the details, see \S \ref{parabolic}.

Our approaches to the above two  geometric results  are similar to the ones used in 
Theorem 4.2 of \cite{Hu05} and also in \S 2.2 of \cite{HMSV06}. For GIT quotients by parabolic 
subgroups, we apply the corresponding results of \cite{Hu06}.

\medskip

In the second half of this paper, we turn our attention to the algebraic aspects of the 
above geometric results. Interestingly, our geometric correspondence finds its natural 
counterpart in representation theory in the context of the {\it reciprocity algebra} studied by
Howe and his collaborators \cite{HL07, HTW08}. For this,  we construct in \S \ref{reciprocity}
an algebra whose homogeneous components provide invariant section spaces as arising in the 
parabolic quotient correspondences. Then we show that each homogeneous component of this 
algebra records two different types of {\it branching rules} for the representations of the 
general linear groups. This is the algebraic version of the geometric quotient correspondence 
stated in the second theorem.

To be more precise, let $P'_H$ and $P'_T$ be the commutator subgroups of $P_H$ and $P_T$ 
respectively. Then we show that 

\begin{theo} [{\protect = Corollary \ref{individual reciprocity}}]
The $P'_H \times P'_T$-invariant subring of the coordinate ring
of the space $\Rep(Q,\bd)$ is a multi-graded ring whose
homogeneous components encode simultaneously
\begin{enumerate}
\item the multiplicities of the $\prod_{h}\GL_{d_{h}}$ modules $\bigotimes_{h}V_h$ \\
in the $\prod_{h,i}\GL_{d_{h}}$ modules $\bigotimes_{h,i} {W_{h,i}}$ and;
\item the multiplicities of the $\prod_{h,i}\GL_{d_{t_{h,i}}}$ modules
$\bigotimes_{h,i}W'_{h,i}$\\
in the $\prod_{h}\GL_{n_{h}}$ modules $\bigotimes_{h}V'_{h}$
\end{enumerate}
where $h \in H$ and $1 \leq i \leq m(h)$. Moreover, $V_h$ and $V'_h$ 
(respectively $W_{h,i}$ and $W'_{h,i}$) are labeled by the same Young diagrams.
\end{theo}

In \S \ref{transfer}, we present the parabolic quotient correspondence stated in the second theorem
as a geometric version of the {\it transfer principle} in the representation theory.



\bigskip


\section{Quivers and reductive quotient correspondence}\lab{reductive}



\begin{say} As in the introduction, 
a quiver is an oriented graph $Q=(Q_0, Q_1, \bh, \bt)$ equipped with a finite ordered set 
of vertices $Q_0$, a set of arrows $Q_1$, and two functions $\bh, \bt$ such that for each arrow 
$a \in Q_1$, $\bh(a) \in Q_0$ is the head and  $\bt(a) \in Q_0$ is the tail. If $Q_0$ contains 
$m>0$ vertices, we may identify $Q_0$ with the set of integers $\{1, \cdots, m\}$.
\end{say}

\begin{say}\lab{Quiver}
Fix  a vector $\bd=(d_1,\cdots, d_m) \in {\mathbb N}^m$. The representation spaces  of the quiver $Q$ with 
the fixed dimension vector $\bd$ can be, upon choosing bases of relevant vector spaces, identified with
$$\Rep (Q, \bd) = \bigoplus_{a \in Q_1} \mat_{d_{\bh(a)}, d_{\bt(a)}}$$
where $\mat_{p,q}$ is the space of matrices of size $p \times q$.
Let
$$\GL_\bd = \prod_{i \in Q_0}  \GL_{d_i}.$$
Then the reductive group $\GL_\bd $ acts on $\Rep (Q, \bd)$ by conjugation
\beq \lab{action} 
(g_1, \cdots, g_m): (M_a)_{a \in Q_1} \to (g_{\bh(a)} \cdot M_a \cdot  g^{-1}_{\bt(a)}).
\eeq
\end{say}


\begin{defi} \lab{fence} 
 A quiver $Q=(Q_0,Q_1, \bh, \bt)$ is of fence type if there is a disjoint union 
$Q_0=H \sqcup T$ such that $H$ consists of vertices that are heads of arrows and $T$ consists of 
vertices that are tails of arrows. In particular, there are no arrows between any two vertices in 
$H$ {\rm (}$T$, respectively{\rm )}.
\end{defi}

\begin{say}
The reductive group $\GL_\bd = G_H \times G_T$ acts on the affine space $\Rep (Q, \bd)$.
Let $\LL$ be the trivial line bundle $\Rep (Q, \bd) \times \CC$. A character
$$\chi: \GL_\bd \lra \CC^*$$
defines a linearization of $\GL_\bd$ on $\LL$ by
$$g \cdot (M,z) = (g \cdot M, \chi(g) z).$$

We can identify the center of $G_H$ ($G_T$) with $(\CC^*)^{|H|}$ ($(\CC^*)^{|T|}$).
Any character of $G_H$ , respectively $G_T$,  is of the form
\beq\lab{characters} 
\chi_\br ((g_h)_{h \in H} )= \prod_{h \in H} \det (g_h)^{r_h} \;\;
\hbox{respectively,} \; \; \chi_\ee ( (g_t)_{t \in T} )= \prod_{t \in T} \det (g_t)^{e_t}  
\eeq
where $\br=(r_h)_{h \in H} \in \N^{|H|}$ and $\ee=(e_t)_{t \in T} \in \N^{|T|}$. The product
$\chi_{\br} \chi_{-\ee}$ defines a character of $\GL_\bd=G_H \times G_T$ and any character of 
$\GL_\bd$ is of this form. We let $\LL(\chi_{\br} \chi_{-\ee})$ be the associated linearized line 
bundle. We introduce 
$$K =\left\{(z_vI_{d_v})_{v \in Q_0} \in  \GL_\bd\; | \;z_v \in \CC^*, \; z_{\bh(a)}=z_{\bt(a)}, \; \forall a \in Q_1\right\}.$$
Then the subgroup $K$ acts trivially on $\Rep(Q,\bd)$. We let $\GL_\bd'= \GL_\bd/K$ and consider 
the action of this quotient group. The character $\chi_{\br} \chi_{-\ee}$ descends to a character of 
$\GL'_\bd$ if and only if
\beq \lab{rd=ed}
\sum_{h \in H}  r_h \cdot  d_h = \sum_{t \in T} e_t \cdot d_t.
\eeq
\end{say}

\begin{say}  
For any $h \in H$ and $t \in T$, we set
\beq\lab{nv} n_h = \sum_{a \in Q_1, \bh(a)=h} d_{\bt(a)} \and n_t=\sum_{a \in Q_1, \bt(a)=t} d_{\bh(a)}.\eeq
We also let
$$X_T = \prod_{t \in T} \Gr(d_t, \CC^{n_t}) \and X_H= \prod_{h \in H} \Gr(d_h, \CC^{n_h}).$$
The group $G_H$ acts on $X_T$ naturally; the group $G_T$ acts on $X_H$ naturally.
We consider the action of $G_H$ on $X_T$ first. We let
$$L_\ee =\boxtimes_{t \in T} \sO_{ \Gr(d_t, \CC^{n_t})} (e_t)$$
be the line bundle over $X_T$ determined by $\ee$. Then the character $\chi_\br: G_H \to \CC^*$  of 
\eqref {characters} defines a $G_H$-linearization $L_\ee(\br)$ over $X_T$.  Similarly,
we let $$L_\br =\boxtimes_{h \in H} \sO_{ \Gr(d_h, \CC^{n_h})} (r_h)$$ be the line bundle over 
$X_H$ determined by $\br$. Then the character $\chi_\ee: G_T \to \CC^*$ of \eqref {characters}
defines a $G_T$-linearization $L_\br(\ee)$ over $X_H$.  
\end{say}

\begin{say} Now we prove Theorem \ref{reductiveGM2}.
\begin{proof} First, we can write
$$\Rep(Q,\bd)= \bigoplus_{t \in T} \bigoplus_{a \in Q_1, \bt(a)=t} \mat_{d_{\bh(a),d_t}}.$$
We identify  $\bigoplus_{a \in Q_1, \bt(a)=t} \mat_{d_{\bh(a),d_t}}$ with
$\mat_{n_t,d_t}$ by placing individual matrices of $\mat_{d_{\bh(a),d_t}}$ in rows.
For simplicity, we shall write $\LL$ for $\LL(\chi_{\br} \chi_{-\ee})$.
Then by applying the first fundamental theorem of invariant theory to every individual factor
$\Gr(d_t, \CC^{n_t})$ of $X_T$,  we have that for any $N \ge 0$,
$$\Ga(X_T, L_\ee^N)= \Ga(\Rep(Q,\bd), \LL^{(\sum e_t d_t N)})^{G_T} .$$
Consequently, we have
$$\Ga(X_T, L_\ee^N)^{G_H}= \Ga(\Rep(Q,\bd), \LL^{(\sum e_t d_t N)})^{\GL_\bd} .$$
Likewise, we can also write
$$\Rep(Q,\bd)= \bigoplus_{h \in H} \bigoplus_{a \in Q_1, \bh(a)=h} \mat_{d_h, d_{\bt(a)}}.$$
We identify $\bigoplus_{a \in Q_1, \bh(a)=h} \mat_{d_h, d_{\bt(a)}}$ with $\mat_{d_h, n_h}$ 
by placing matrices of $\mat_{d_h, d_{\bt(a)}}$ in different columns. Then by the first 
fundamental theorem of invariant theory, we obtain
$$\Ga(X_H, L_\br^N)= \Ga(\Rep(Q,\bd), \LL^{(\sum r_h d_h N)})^{G_H}, $$  hence
$$\Ga(X_H, L_\br^N)^{G_T}= \Ga(\Rep(Q,\bd), \LL^{(\sum r_h d_h N)})^{\GL_\bd}.$$
Because $\sum_{t \in T} e_t d_t = \sum_{h \in H} r_h d_h,$ we see that
$$\Ga(X_T, L_\ee^N)^{G_H}= \Ga(\Rep(Q,\bd), \LL^{(\sum r_h d_h N)})^{\GL_\bd} = \Ga(X_H, L_\br^N)^{G_T}.$$
This implies that we have natural isomorphisms of GIT quotients
$$X_T^{ss}(L_\ee(\br))/\!/G_H \cong \Rep(Q,\bd)^{ss}(\LL(\chi_{\br} \chi_{-\ee}))/\!/{\GL_\bd} \cong 
X_H^{ss}(L_\br(\ee))/\!/G_T.$$
\end{proof}
\end{say}

\begin{say} 
Let $\vartheta=(\br, -\ee) \in \ZZ ^{|Q_0|} $. By \eqref{rd=ed}, $\bd \cdot \vartheta=0$. 
Here, ``$\cdot$'' is the canonical dot product in $ \ZZ ^{|Q_0|}$.  By King \cite{King},
a quiver representation $(V_i)_{i \in Q_0}$ is $\LL(\chi_{\br} \chi_{-\ee})$-semistable 
if and only if for all non-trivial proper subrepresentation $(E_i)_{i \in Q_0}$, 
\beq \lab{king} 
(\dim E_i) \cdot \vartheta \le 0.
\eeq
Using this, we can give a stability criterion for the action of $G_T$ on $X_H$.
Let $(V_i)_{i \in H} \in X_H= \prod_{h \in H} \Gr(d_h, \CC^{n_h})$ be any point. Note that
for each $h \in H$, we have a fixed decomposition
$$ \CC^{n_h} = \oplus_{a \in Q_1(h)} \CC^{d_{\bt(a)}}$$
where $Q_1(h)= \{a \in Q_1, \bh(a)=h \}$. Then $(V_i)_{i \in H} $ is semistable with respect to 
$L_\br(\ee)$ if and only if for all $(E_i)_{i \in H}$ with $E_i$ a subspace of $V_i$, we have
\beq \lab{king2} 
\sum_{i \in H} r_i \dim E_i \leq \sum_{i \in H} \sum_{a \in Q_1(i)} e_{\bt(a)} \dim (E_{\bt(a)} 
\cap \mathbb{C}^{d_{\bt(a)}}). 
\eeq
Likewise, for any point  $(V_j)_{j \in T} \in X_T= \prod_{t \in T} \Gr(d_t, \CC^{n_t})$,
it is semistable with respect to $L_\ee(\br)$ if and only if for all $(E_j)_{j \in T}$ with $E_j$ 
a subspace of $V_j$, we have
\beq \lab{king3}  
\sum_{j \in T} \sum_{a \in Q_1(j)} r_{\bh(a)} \dim (E_{\bh(a)} \cap \mathbb{C}^{d_{\bh(a)}}) 
 \leq \sum_{j \in T} e_j \dim E_j. 
\eeq

Note that,  replacing $\leq$ by $<$ in \eqref{king}, \eqref{king2} and \eqref{king3}, 
we obtain the conditions for stable representations.
\end{say}

\begin{say} 
Note that either of the GIT quotients
 $$X^{ss}(L_\ee(\br))/\!/G_H  \and Y^{ss}(L_\br(\ee))/\!/G_T$$ is 
the quiver moduli determined by $\vartheta=(\br, -\ee)$ 
for the fence quiver (\ref{fence}) with the given dimension vector $\bd$.  
It would be an interesting problem  to find other classes of quivers and dimension vectors 
such that their quiver varieties have the kind of geometric interpretations as 
in Theorem \ref{reductiveGM2}.
\end{say}

\begin{say} \lab{starQuiver} As examples, we now revisit the GM correspondence of \cite{Hu05} using
the language of quivers. For this, we let $Q$ be the star quiver with vertices $Q_0=\{0, 1, \cdots, m\}$
with $$Q_1=\{ 1\to 0, \cdots, m\to 0\}$$ (0 is the unique head for all arrows).  This is a special case of fence quivers
where $H$ consists of a single vertex.  Let
$$\bd = (n, k_1, \cdots, k_m) \in {\mathbb N}^{m+1}$$ such that for each $1 \le i \le m$,
$ k_i \le n \le \sum_i k_i$.  In this case, we have
$$\Rep (Q, \bd)=\bigoplus_{i=1}^m \mat_{n,k_i}, \quad \GL_\bd = \GL_n \times \prod_{i=1}^m \GL_{k_i}$$
where as in \eqref{action}, $\GL_n$ acts by left multiplication and $ \prod_{i=1}^m\GL_{k_i}$ acts on the right 
by the inverse multiplication, component-wise.
%
%
We have in this case
$$X_T= \prod_{i=1}^m \Gr(k_i, \CC^n) \and X_H= \Gr(n, \CC^k) $$
where $k=\sum_{i=1}^m k_i$.  The group $\GL_n$ acts on $X_T$ naturally and the group
$ \prod_{i=1}^m\GL_{k_i}$ acts on $X_H$ in the natural way. Let $r \in \N$ and $\ee \in \N^m$ such that
\beq r \cdot n = \sum_i e_i \cdot k_i. \eeq
Then as a special case of Theorem \ref{reductiveGM} , we have
\end{say}

\begin{coro}\lab{reductiveGM} {\rm (\cite{Hu05})}
There is a natural one-to-one correspondence between the set of GIT quotients of $\prod_{i=1}^m \Gr(k_i, \CC^n)$ 
by $\GL_n$ and the set of GIT quotients of  $ \Gr(n, \CC^k)$  by $ \prod_{i=1}^m\GL_{k_i}$. Precisely,
suppose that $\br \in \N$ and $\ee \in \N^m$ satisfy the compatibility condition $rn = \sum_i e_i k_i$,
then we have a natural isomorphism between $ \prod_{i=1}^m \Gr(k_i, \CC^n)(L_\ee(\br))/\!/\GL_n$ and
$ \Gr(n, \CC^k)^{ss}(L_\br(\ee))/\!/\prod_{i=1}^m\GL_{k_i}$.\end{coro}

\bigskip


\section{Parabolic Quivers and  Correspondence}\lab{parabolic}


\begin{defi} 
Let $Q=(Q_0,Q_1,\bh,\bt)$ be a quiver. A representation of $Q$ with parabolic structures 
is a representation of the quiver $Q$ together with some (partial) flags of $V_b$
at every vertex $b \in Q_0$.
\end{defi}

\begin{say}
At some vertices, the partial flags may be trivial. When all are trivial, 
we have an ordinary quiver representation.
\end{say}

\begin{say}  
To specify the flags,  for any vertex $v \in Q_0$, we fix a partition 
$$d_v= d_{v_1} + \cdots +d_{v_s}$$ of $d_v$ by positive integers
where $s=s(v)$ is a positive integer depending on the vertex $v$. We
let  $P_v$ be the parabolic subgroup of $\GL_{d_v}$  consisting of 
the block upper triangular matrices whose
diagonal blocks are of the size $(d_{v_1}, \cdots , d_{v_s})$
\beq\label{Parabolic}
P_v=\left\{ \left[
\begin{array}{cccc}
B_{11} & B_{12} & \cdots & B_{1s} \\
& B_{22} &  & \vdots \\
&  & \ddots &  \\
0 &  &  & B_{ss}%
\end{array}%
\right] \right\}
\eeq%
where $B_{ii}\in GL_{d_{v_i}}$ for $1 \leq i \leq s$.
\end{say}

\begin{say}
Associated to each $v \in Q_0$, we define the following (partial) flag variety
\begin{equation*}\label{flag variety}
Y_v:= Fl(d_{v_1}, \cdots , d_{v_s};\mathbb{C}^{n_v})=\left\{ \left( 0\subset V_{1}\subset
\cdots \subset V_{s}\subset \mathbb{C}^{n_v}\right) :\dim
V_{i}=\sum_{j=1}^{i}d_{v_j}\right\}
\end{equation*}
where $n_v$ is as define in \eqref{nv}.
For every $v \in Q_0$,  an $s$-tuple $\br =(r_{1},\cdots ,r_{s})\in \mathbb{N}^s$ defines
 a (very) ample line bundle over $Y_v$:
\begin{equation*}
L_{\mathbf{r}}=\sO_{Gr(d_{v_1}, \CC^{n_v})}(r_{1})\otimes \sO_{Gr(d_{v_1}+d_{v_2}, \CC^{n_v})}(r_{2})
\otimes \cdots \otimes\sO_{Gr(d_v, \CC^{n_v})}(r_s)
\end{equation*}%
(recall here that $s=s(v)$ depends on $v$).
This line bundle is induced from the Pl\"ucker embedding
of the flag variety  $Y_v$ into the product of the Grassmannian
$\prod_{i}\Gr(\sum_{j=1}^{i}d_{v_j},\mathbb{C}^{n_v}\mathbb{)}$.
The action of the reductive group $\GL_{d_v}$ lifts canonically to
$L_{\br}$, making it a $\GL_{d_v}$-linearized line bundle.
The parabolic subgroup $P_v$ inherits  the action.
\end{say}

\begin{say}  The parabolic subgroup $P_v$ has the group of characters of dimension $s$,
we can twist the $P_v$-linearized line bundle $L_{\mathbf{r}}$ by its characters.  For this, note that
 the commutator subgroup $P_v^{\prime }$ of $P_v$
consists of elements with $B_{ii}\in SL_{d_{v_i}}$ for each $i$ as in
(\ref{Parabolic}).  Therefore,  we have $P_v/P_v^{\prime }\cong \left( \mathbb{C}^{\ast }\right) ^{s}$
and this can be identified with the center of $P_v$ by
\begin{eqnarray} \label{P-character1}
(\tau _{1},\cdots ,\tau _{s}) &\longmapsto &\left[
\begin{array}{cccc}
\tau _{1}I_{d_{v_1}} &  &  & 0 \\
& \tau _{2}I_{d_{v_2}} &  &  \\
&  & \ddots &  \\
0 &  &  & \tau _{s}I_{d_{v_s}}%
\end{array}%
\right] \in P_v.
\end{eqnarray}
Now for any $s$-tuple of positive integers
$$\ee=(e_{1},\cdots, e_{s})\in \mathbb{N}^{s},$$
it defines a character $\mu_\ee$ of $P_v$
\begin{eqnarray}\label{P-character2}
\mu _{\mathbf{e}}:P_v &\rightarrow &\mathbb{C}^{\ast } \\
\mu _{\mathbf{e}}(\tau _{1},\cdots ,\tau _{s}) &=&\tau
_{1}^{d_{v_1}(e_{1}+e_{2}+\cdots +e_{s})}\tau _{2}^{d_{v_2}(e_{2}+e_{3}+\cdots
+e_{s})}\cdots \tau _{s}^{d_{v_s}e_{s}}. \notag
\end{eqnarray}
Then the character defines a $P_v$-linearized line bundle $L_{\br} (\ee)$ over $Y_v$.
We should make a useful note here:
both $\br$ and $\ee$ are some chosen $s$-tuples of positive integers with $s=s(v)$ depending on $v$;
$\br$ defines a line bundle $L_\br$; $\ee$ defines a character $\mu_\ee$. Of course, the roles of $\br$ and $\ee$ can
be switched.
\end{say}

\begin{say}
Now we set
$$Y_T= \prod_{t \in T} Y_t \and Y_H = \prod_{h \in H} Y_h.$$
We also let
$$P_H= \prod_{h \in H} P_h \and P_T = \prod_{t \in T} P_t.$$
Then $P_H$ acts on $Y_T$ naturally and $P_T$ acts on $Y_H$ naturally.
Suppose that for every $v \in Q_0$, we have chosen a pair of $s(v)$-tuples
$\br (v)$ and $\ee (v)$ in $\N^{s(v)}$.  Then over $Y_T$, we have an induced ample line bundle
$$L_{\br_T} = \boxtimes_{t \in T} L_{\br(t)}.$$ For any $h \in H$, $\ee(h)$ defines a character of
$P_h$, hence their product induces a character $\ee_H$ of $P_H$. This gives rise to a
$P_H$-linearized line bundle $L_{\br_T}(\ee_H)$. Thus we have the Zariski open subset
$Y_T^{ss}(L_{\br_T}(\ee_H))$  and its quotient
$Y_T^{ss}(L_{\br_T}(\ee_H))/\!/P_H$
(see \cite{Hu06} for a quotient theory of parabolic subgroup actions).

Likewise, we have the induced ample line bundle
$$L_{\ee_H} = \boxtimes_{h \in H} L_{\ee(h)}$$
over $Y_H$. Then the product of all characters $\br (t)$ for all $t \in T$ gives rise to
a character $\br_T$  of $P_T$ and hence a
$P_T$-linearized line bundle $L_{\ee_H}(\br_T)$. Thus we have the Zariski open subset
$Y_H^{ss}(L_{\ee_H}(\br_T))$  and its quotient
$Y_H^{ss}(L_{\ee_H}(\br_T))/\!/P_T.$
 \end{say}

\begin{say} We are now almost  ready to state our main geometric theorem of this section.
Before making the statement, for the similar reason as in \eqref{rd=ed}, we need to impose a condition.
 \begin{equation}\lab{cond}
 \sum_{t \in T}\sum_{1\leq i\leq s(t)} d_{t_i} (r_{i}+\cdots +r_{s(t)})=
 \sum_{h \in H}\sum_{1\leq i\leq s(h)} d_{h_i} (e_{i}+\cdots +e_{s(h)}).
\end{equation}

\end{say}

\begin{theo}  \lab{parabolicGM}
There is an one-to-one correspondence between the set of GIT quotients of
$Y_T$ by $P_H$ and the set of GIT quotients of $Y_H$ by $P_T$. More precisely,
assume that for all $v \in Q_0$, the chosen pairs
$\br(v)$ and $\ee(v)$ in $\N^{s(v)}$ satisfy  (\ref{cond}). Then we have a natural
isomorphism between the quotient
$Y_T^{ss}(L_{\br_T}(\ee_H))/\!/P_H$  and the quotient
$Y_H^{ss}(L_{\ee_H}(\br_T))/\!/P_T.$
\end{theo}

\begin{proof} As in the proof of Theorem \ref{reductiveGM2}, we  first identify $\Rep(Q,\bd)$ with
$$\bigoplus_{t \in T} \mat_{n_t,d_t}.$$
 Then for any $N>0,$ by \cite[\S 9]{Fu97}, we have
 \begin{equation*}
\Gamma (Y_T, L_{\br_T}^{\otimes N}) \cong \Gamma (\Rep(Q,\bd),
\LL^{\otimes bN})^{P_T}
\end{equation*}
where $b=\sum_{t \in T}\sum_{1\leq i\leq s(t)} d_{t_i} (r_{i}+\cdots +r_{s(t)})$.
Likewise, we have
 \begin{equation*}
\Gamma (Y_H, L_{\br_H}^{\otimes N}) \cong \Gamma (\Rep(Q,\bd),
\LL^{\otimes cN})^{P_H}
\end{equation*}
$c= \sum_{h \in H}\sum_{1\leq i\leq s(h)} d_{h_i} (r_{i}+\cdots +r_{s(h)})$. Note that $b=c$. Hence, we have
\begin{equation*}
\Gamma (Y_T, L_{\br_T}(\ee_H)^{\otimes N})^{P_H} \cong \Gamma (\Rep(Q,\bd),
\LL^{\otimes bN})^{P_H \times P_T} \cong \Gamma (Y_H, L_{\br_H}^{\otimes N})^{P_T}. \end{equation*}
 \medskip
This implies the natural isomorphisms of the quotients
$$Y_T^{ss}(L_{\br_T}(\ee_H))/\!/P_H \cong \Rep(Q,\bd)^{ss}(\br,\ee)/\!/(P_H \times P_T)
\cong Y_H^{ss}(L_{\ee_H}(\br_T))/\!/P_T.$$ \end{proof}

\begin{rema}
Note that the quotient $\Rep(Q,\bd)^{ss}(\br,\ee)/\!/(P_H \times P_T)$ parameterizes
the equivalence classes of semistable parabolic quivers. \end{rema}



\begin{say} As a special case,  
here we assume that all the parabolic structures on tails are trivial. In such a case,
Theorem \ref{parabolicGM} will specialize to a correspondence between quotients of a reductive group action
and quotients of a parabolic subgroup.
This gives a GM Correspondence of mixed types. To be more precise, in this case, we have
$Y_T=X_T =\prod_{t \in T} \Gr(d_t, \CC^{n_t}) $ is a product of Grassmannians
and $Y_H =  \prod_{h \in H} Y_h $ is a product of flag varieties. Acting on $X_T$ is  the parabolic
group $P_H$; acting on $Y_H$ is the reductive group $G_T$.  The condition \eqref{cond} becomes
 \begin{equation}
 \sum_{t \in T} d_t r_t=
 \sum_{h \in H}\sum_{1\leq i\leq s(h)} d_{h_i} (e_{i}+\cdots +e_{s(h)}).
\end{equation}
Under this condition,  we have a natural
isomorphism between the  quotient
$X_T^{ss}(L_{\br_T}(\ee_H))/\!/P_H$ by parabolic subgroup  and the quotient
$Y_H^{ss}(L_{\ee_H}(\br_T)/\!/G_T$ by reductive group.
(For further correspondence between quotients by a reductive group and quotients by its parabolic subgroups,
consult \cite{Hu06}.)
\end{say}


\begin{say}  We may further specialize to the case when $Q$ is the star quiver
as considered in  \ref{starQuiver}.
In this case, $Y_T$ is $\prod_{t \in T} \Gr(d_t, \CC^n)$ and $Y_H$ is the (single, partial) flag variety  consisting
flags of the type
$$ 0\subset V_{1}\subset
\cdots \subset V_{s}\subset \mathbb{C}^{d}, \quad \dim
V_{i}=\sum_{j=1}^{i}n_{j}$$
where $d= \sum_{t \in T} d_t$ and $n=\sum_{i =1}^s n_i$ is a partition of $n$.
Over the product of the Grassmannians $Y_T$, we have a natural diagonal action of the parabolic group $P_H$;
over the (partial) flag variety $Y_H$, we have the action of $G_T=\prod \GL_{d_t}$.
The condition \eqref{cond} now reads
 \begin{equation}
  \sum_{t \in T} d_t r_t=
 \sum_i n_i (e_{i}+\cdots +e_s).
\end{equation}
  \end{say}

\bigskip


\section{Multi-Reciprocity Algebras}\label{reciprocity}


In this section, we study representation theoretic correspondences matching
with our geometric ones.

\begin{say}
For reductive groups $K$ and $G$ with $K\subset G$, {we} consider
irreducible representations $V$ and $W$ of $K$ and $G${,} respectively. By
Schur's lemma, the multiplicity of $V$ in $W$ as a representation of $K$ is
equal to the dimension of the space
\begin{equation*}
Hom_{K}(V,W),
\end{equation*}%
which is called the \textit{multiplicity space}. The \textit{branching rule}
under the restriction of $G$ down to $K$ is a description of the
multiplicity spaces. {A special case of the above is when $G=K\times \cdots
\times K$ and $K$ is identified with the diagonal subgroup in $G$. In this
case,} the multiplicity spaces
\begin{equation*}
Hom_{K}(V,W_{1}\otimes \cdots \otimes W_{m})
\end{equation*}%
describe the decompositions of tensor products of $K$-modules $W_{i}$. {In
what follows,} we shall describe the invariant section spaces
\begin{equation*}
\Gamma (\Rep(Q,\bd),\LL^{\otimes b})^{P_{H}\times P_{T}}
\end{equation*}%
as graded components of an algebra encoding branching rules of different
types.

\end{say}

\begin{say}

First we recall Young diagrams as a labeling system for irreducible
polynomial representations of $\GL_{n}$. Every irreducible polynomial
representation of $\GL_{n}$ is uniquely labeled by, under the identification
with its highest weight, a Young diagram with no more than $n$ rows. Let $%
\rho _{n}^{F}$ denote the irreducible representation of $\GL_{n}$ labeled by
Young diagram $F=(f_{1},\cdots ,f_{n})\in \mathbb{Z}^{n}$ with $f_{1}\geq
\cdots \geq f_{n}\geq 0$. {Here, $(f_{1},\cdots ,f_{n})$ is identified with
the highest weight of $\rho _{n}^{F}$ .} The dual representation of $\rho
_{n}^{F}$ has the highest weight
\begin{equation*}
F^{\ast }=(-f_{n},\cdots ,-f_{1})
\end{equation*}%
and will be denoted by $\rho _{n}^{F^{\ast }}$. We write $\ell (F)$
for the number of non-zero entries in $F$. If
the entries $f_{i}$ of $F$ repeat $a_{i}$ times, then we also write $%
F=(f_{1}^{a_{1}},f_{2}^{a_{2}},\cdots )$ with $f_{1}>f_{2}>\cdots >0$. For
example, $F=(4,4,3,2,2)$ or $(4^{2},3^{1},2^{2})$ can be drawn as%
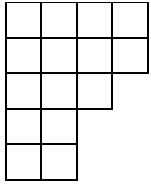
\begin{equation*}
\young(\ \ \ \ ,\ \ \ \ ,\ \ \ ,\ \ ,\ \ )
\end{equation*}%
and $\ell (F)=5.$ Then the Young diagram for $F^{\ast }$ can be drawn
by rotating $F$ around its center by $180 ^\circ$.
\end{say}

\begin{say}
For $(k_{1},\cdots ,k_{s})\in \mathbb{N}^{s}$,
let $k=k_{1}+\cdots +k_{s}$. We let $P_{k}$ denote the parabolic
subgroup of $\GL_{k}$ consisting of the block upper triangular matrices
whose diagonal blocks are of the sizes $k_{1},\cdots ,k_{s}$%
\begin{equation*}
P_{k}=\left\{ \left[
\begin{array}{cccc}
B_{11} & B_{12} & \cdots & B_{1s} \\
& B_{22} &  & \vdots \\
&  & \ddots &  \\
0 &  &  & B_{ss}%
\end{array}%
\right] \right\}
\end{equation*}%
where $B_{ii}\in GL_{k_{i}}$ for $1 \leq i \leq s$.
Note that the commutator subgroup $P_{k}^{\prime }$ of $P_{k}$ contains 
$\mathrm{{SL}_{k_i}}$'s on the diagonal.
\begin{lemm}
\label{P-invariant highest} Let $P_{k}$ be the parabolic subgroup of $GL_{k}$
as given above and $P_{k}^{\prime }$ be its commutator subgroup. Then the
dimension of the $P_{k}^{\prime }$-invariant subspace of $\rho _{k}^{F}$ is
at most $1$. It {equals} $1$ if and only if $F$ is of the form %
$$(f_{1}^{k_{1}},f_{2}^{k_{2}},\cdots ,f_{s}^{k_{s}}) \in \mathbb{N}^{s}$$
with $f_{1}>\cdots >f_{s}>0$. In particular, if $s=1$ then $k=k_1$ and the dimension of
$(\rho _{k}^{F})^{\mathrm{{SL}_{k}}}$ is $1$ only when $F$ is of the
form $(f_{1}^{k})$, i.e., a rectangular diagram with $k$ rows.
\end{lemm}

\begin{proof}
Note that $P_{k}^{\prime }$ contains the maximal unipotent subgroup $U_{k}$
of $\GL_{k}$. By highest weight theory (e.g., \cite[\S 3]{GW09}), $(\rho
_{k}^{F})^{U_{k}}$ is the one dimensional subspace spanned by a highest
weight vector of $\rho _{k}^{F}$. Therefore, the dimension of $(\rho
_{k}^{F})^{P_{k}^{\prime }}$ is less than or equal to $1$. The condition for
this being exactly $1$ can be obtained by simply computing invariants under
the diagonal blocks of $P_{k}^{\prime }$ or can be found in \cite[\S 9.3]{Fu97}.
\end{proof}

The same statement holds also for the dual representations $\rho
_{k}^{F^{\ast }}$ and $F^{\ast }$.
\end{say}

\begin{say}
As we noted in (\ref{P-character1}) we have
$P_{k}/P_{k}^{\prime }\cong \left( \mathbb{C}^{\ast }\right) ^{s}$.
Then for Young diagram $F=(f_{1}^{k_{1}},\cdots ,f_{s}^{k_{s}})$,
$P_{k}/P_{k}^{\prime }$ is acting on the one dimensional space $(\rho
_{k}^{F})^{P_{k}^{\prime }}$ via the character
\begin{eqnarray*}
& \mu _{F} :\left( \mathbb{C}^{\ast }\right) ^{s}\rightarrow \mathbb{C}%
^{\ast } \\
& \mu _{F}(\tau _{1},\cdots ,\tau _{s}) = \tau _{1}^{k_{1}f_{1}}\cdots \tau
_{s}^{k_{s}f_{s}}
\end{eqnarray*}%
This notation is the same as our previous one $\mu _{\mathbf{e}}$ for $%
\mathbf{e}=(e_{1},\cdots ,e_{s})\in \mathbb{N}^{s}$ given in (\ref{P-character2})
by setting $f_{i}=e_{i}+\cdots +e_{s}$ for $1\leq i\leq s$. We let  $A_{k}^{+}$
denote the semigroup of the characters of $P_{k}$\:%
\begin{equation*}
A_{k}^{+}=\{\mu _{F}:F=(f_{1}^{k_{1}},\cdots ,f_{s}^{k_{s}})\}.
\end{equation*}
\end{say}

\begin{say}
Now we give an action of $\GL_{n}\times \GL_{k}$ on the space $\mat_{n{,}k}$ by
\begin{equation*}
(g_{1},g_{2})\cdot M=g_{1}Mg_{2}^{-1}
\end{equation*}%
for $(g_{1},g_{2})\in \GL_{n}\times \GL_{k}$ and $M\in \mat_{n,k}$.
\end{say}

\begin{lemm}\label{GLm-GLn}
\begin{enumerate}
\item With respect to the action of $\GL_{n}\times \GL_{k}$, the coordinate
algebra\ $\mathbb{C}[\mat_{n,k}]$ of $\mat_{n,k}$ decomposes as%
\begin{equation*}
\mathbb{C}[\mat_{n,k}]=\sum\limits_{\ell (F)\leq \min (n,k)}\rho
_{n}^{F^{\ast }}\otimes \rho _{k}^{F}
\end{equation*}%
where the sum runs over all $F$ with less than or equal to $\min (n,k)$
rows.

\item For a parabolic subgroup $P_{k}$ of $\GL_{k}$, the $P_{k}^{\prime }$%
-invariant subalgebra of $\mathbb{C}[\mat_{n,k}]$ is graded by the semigroup
$A_{k}^{+}$ and has decomposition%
\begin{equation*}
\mathbb{C}[\mat_{n,k}]^{P_{k}^{\prime }}=\sum\limits_{F}\rho _{n}^{F^{\ast }}
\end{equation*}%
over $F$ such that $\ell (F)\leq \min (n,k)$ and $\dim \left( \rho
_{k}^{F}\right) ^{P_{k}^{\prime }}=1$.
\end{enumerate}
\end{lemm}

\begin{proof}
The first statement is known as the $\GL_{n}$-$\GL_{k}$ duality
(e.g., \cite[Theorem 5.6.7]{GW09}). For the second statement, by taking
$\left( 1\times P_{k}^{\prime }\right) $-invariant subalgebra of
$\mathbb{C}[\mat_{n,k}]$, we have%
\begin{equation*}
\mathbb{C}[\mat_{nk}]^{P_{k}^{\prime }}=\sum\limits_{\ell (F)\leq \min
(n,k)}\rho _{n}^{F^{\ast }}\otimes \left( \rho _{k}^{F}\right)
^{P_{k}^{\prime }}
\end{equation*}%
Then for each $F$, from Lemma \ref{P-invariant highest}, the space $\rho
_{n}^{F^{\ast }}\otimes \left( \rho _{k}^{F}\right) ^{P_{k}^{\prime }}$ is
nonzero exactly when $\dim \left( \rho _{k}^{F}\right) ^{P_{k}^{\prime }}=1$
and in this case the space
\begin{equation*}
\rho _{n}^{F^{\ast }}\otimes \left( \rho _{k}^{F}\right) ^{P_{k}^{\prime
}}\cong \rho _{n}^{F^{\ast }}
\end{equation*}%
is the $A_{k}^{+}$-eigenspace with weight $\mu _{F}$. Hence $\mathbb{C}[\mat%
_{nk}]^{P_{k}^{\prime }}$ is the sum of $A_{k}^{+}$-eigenspaces, and this
gives the grading structure.
\end{proof}

\begin{say}
Let us begin with the coordinate algebra $\mathbb{C}[\Rep(Q,\bd)]$ of the
representation space of a fence quiver $Q=(Q_{0},Q_{1},\bh,\bt)$ with the
dimension vector $\bd$. We shall use the same notation introduced in
Section \ref{parabolic}.

By using the identification of $\Rep(Q,\bd)$ with
the direct sum of the spaces $\mat_{d_{\bh(a)},d_{\bt(a)}}$, we have%
\begin{eqnarray}
\mathbb{C}[\Rep(Q,\bd)] &=&\bigotimes_{a\in Q_{1}}\mathbb{C}[\mat%
_{d_{\bh(a)},d_{\bt(a)}}]  \label{REP-Ring1} \\
&=&\bigotimes_{h\in H}\bigotimes_{a\in Q_{1}(h)}\mathbb{C}[\mat%
_{d_{h},d_{\bt(a)}}]  \notag \\
&=&\bigotimes_{h\in H}\mathbb{C}[\mat_{d_{h},{n}_{h}}]  \notag
\end{eqnarray}%
where $Q_{1}(h)=\{a\in Q_{1}:\bh(a)=h \}$ and ${n}_{h}=\sum_{a \in Q_{1}(h)} d_{\bt(a)}$.

\medskip

To be more precise, for each $h\in H$, we set
\begin{equation*}
\left\{\bt(a):a\in Q_{1}(h)\right\}=\left\{t_{h,1},\cdots ,t_{h,m(h)}\right\}.
\end{equation*}%
If $h$ is clear from the context, then we set
\begin{equation*}
t[i]=t_{h,i}
\end{equation*}
for $1 \leq i \leq m(h)$. We shall show that for each $h \in H$, 
the ${(}{P}_{d_{h}}^{\prime }\times \prod {P}_{d_{t[i]}}^{\prime }{)}$-invariant subalgebra 
of the algebra\ $\mathbb{C}[\mat_{d_{h},{n}_{h}}]$ records two different types of branching
rules with respect to the following restrictions:
\begin{eqnarray*}
&&\GL_{d_{h}}\subset \GL_{d_{h}}\times \cdots \times \GL_{d_{h}} \hbox{ ($m(h)$ copies)}; \\
&&\GL_{d_{t[1]}}\times \cdots \times \GL_{d_{t[m(h)]}} \subset \GL_{{n}_{h}}.
\end{eqnarray*}%

Then by iterating this result over $h\in H$, we can obtain a complete description of
$\left( P^{\prime }_{H}\times P^{\prime }_{T}\right) $-invariant subalgebra of
$\mathbb{C}[\Rep(Q,\bd)]$ to show the following theorem.
\end{say}

%



%
\begin{theo} \label{multi-reciprocity}
\begin{enumerate}
\item The algebra $\mathbb{C}[\Rep(Q,\bd)]^{P_{H}^{\prime }\times P_{T}^{\prime }}$ is graded 
by $$\prod_{h}\left( {A}_{d_{h}}^{+}{\times }\prod_{i=1}^{m(h)}{A}_{d_{t[i]}}^{+}\right).$$

\item For each $h\in H$, we consider Young diagrams $F(h)$ and $D(h,i)$ such
that
\begin{equation*}
\dim \left( \rho _{d_{h}}^{F(h)}\right) ^{{P}_{d_{h}}^{\prime }} = 
\dim \left( \rho _{d_{t[i]}}^{D(h,i)}\right) ^{{P}_{d_{t[i]}}^{\prime }} = 1
\end{equation*}%
for $1\leq i\leq m(h)$. \\
The dimension of the $\prod_{h}(\mu _{F(h)},\mu
_{D(h,1)},\cdots ,\mu _{D(h,m(h))})$-homogeneous component for the algebra $%
\mathbb{C}[\Rep(Q,\bd)]^{P_{H}^{\prime }\times P_{T}^{\prime }}$ records
simultaneously

\begin{enumerate}
\item the multiplicity of the $\prod_{h}\GL_{d_{h}}$-module $%
\bigotimes_{h}\rho _{d_{h}}^{F(h)}$ in
\begin{equation*}
\bigotimes_{h}\left( \rho _{d_{h}}^{D(h,1)}\otimes \cdots \otimes \rho
_{d_{h}}^{D(h,m(h))}\right) ;
\end{equation*}

\item the multiplicity of the $\prod_{h}\prod_{i}\GL_{d_{t_{h,i}}}$-module
\begin{equation*}
\bigotimes_{h}\left( \rho _{d_{t[1]}}^{D(h,1)}\otimes \cdots \otimes 
\rho_{d_{t[m(h)]}}^{D(h,m(h))}\right)
\end{equation*}%
in $\bigotimes_{h}\rho _{n_{h}}^{F(h)}$.
\end{enumerate}
\end{enumerate}
\end{theo}

\begin{say}
To prove this theorem, we investigate the individual
components $\mathbb{C}[\mat_{d_{h},{n}_{h}}]$ of $\mathbb{C}[\Rep(Q,\bd)]$
given in (\ref{REP-Ring1}). Then the theorem can be obtained simply by repeating 
Corollary \ref{individual reciprocity} on $\mathbb{C}[\mat_{d_{h},{n}_{h}}]$ over $h\in H$.

Now let us fix $h$ and then write $n$ for $d_{h}$, and $c_{i}$ for $d_{t[i]}$ for $1 \leq i \leq m=m(h)$. 
Also, let $\underline{c}=(c_{1},\cdots ,c_{m})$ and $c=c_{1}+\cdots +c_{m}$, and set
\begin{eqnarray*}
\GL_{\underline{c}} &=&\prod_{i}\GL_{c_{i}}\and\GL_{\underline{n}}=\prod \GL%
_{n}\hbox{ ($m$ copies)} \\
{P}_{\underline{c}}^{\prime } &=&\prod_{i}{P}_{c_{i}}^{\prime }\and{A}_{%
\underline{c}}^{+}=\prod_{i}{A}_{c_{i}}^{+}
\end{eqnarray*}
\end{say}

\begin{prop}
The following is an $\left( {A}_{n}^{+}{\times A}_{\underline{c}}^{+}\right)
$-graded algebra decomposition of
$\mathbb{C}[\mat_{nk}]^{{P}_{n}^{\prime
}\times {P}_{\underline{c}}^{\prime }}$
\begin{equation*}
\sum\limits_{(D_{1},\cdots ,D_{m})}\sum\limits_{F}\Hom_{\GL_{n}}(\rho
_{n}^{F},\rho _{n}^{D_{1}}\otimes \cdots \otimes \rho _{n}^{D_{m}})\otimes
\left( \rho _{n}^{F^{\ast }}\right) ^{{P}_{n}^{\prime }}
\end{equation*}%
where the sum runs over $F$ and $D_{i}$ such that $\ell (F)\leq \min (n,c)$,
$\ell (D_{i})\leq \min (n,c_{i})$, and
\begin{equation*}
\dim \left( \rho _{n}^{F}\right) ^{{P}_{n}^{\prime }}=\dim \left( \rho
_{c_{i}}^{D_{i}}\right) ^{{P}_{c_{i}}^{\prime }}=1
\end{equation*}
for $1\leq i\leq m$. Each graded component tells us how a $\GL_{\underline{n}}$%
-irreducible representation decomposes as a $\GL_{n}$-module.
\end{prop}

\begin{proof}
By repeating the $\GL_{n}$-$\GL_{c_{i}}$ dualities to the blocks of $\mat%
_{n,c}=\mat_{n,c_{1}}\oplus \cdots \oplus \mat_{n,c_{m}}$, we obtain%
\begin{eqnarray*}
\mathbb{C}[\mat_{n,c}] &=&\mathbb{C}[\mat_{n,c_{1}}]\otimes \cdots \otimes
\mathbb{C}[\mat_{n,c_{m}}] \\
&=&\sum\limits_{(D_{1},\cdots ,D_{m})}\left( \rho _{n}^{D_{1}^{\ast
}}\otimes \rho _{c_{1}}^{D_{1}}\right) \otimes \cdots \otimes \left( \rho
_{n}^{D_{m}^{\ast }}\otimes \rho _{c_{m}}^{D_{m}}\right) \\
&=&\sum\limits_{(D_{1},\cdots ,D_{m})}\left( \rho _{n}^{D_{1}^{\ast
}}\otimes \cdots \otimes \rho _{n}^{D_{m}^{\ast }}\right) \otimes \left(
\rho _{c_{1}}^{D_{1}}\otimes \cdots \otimes \rho _{c_{m}}^{D_{m}}\right)
\end{eqnarray*}%
where the sum runs over all $m$-tuples $(D_{1},\cdots ,D_{m})$ with $%
\ell (D_{i})\leq \min (n,c_{i})$ for each $i$. Then the ${P}_{\underline{c}%
}^{\prime }$-invariants give us%
\begin{equation*}
\mathbb{C}[\mat_{n,c}]^{{P}_{\underline{c}}^{\prime
}}=\sum\limits_{(D_{1},\cdots ,D_{m})}\left( \rho _{n}^{D_{1}^{\ast
}}\otimes \cdots \otimes \rho _{n}^{D_{m}^{\ast }}\right) \otimes \left(
\rho _{c_{1}}^{D_{1}}\right) ^{{P}_{c_{1}}^{\prime }}\otimes \cdots \otimes
\left( \rho _{c_{m}}^{D_{m}}\right) ^{{P}_{c_{m}}^{\prime }}.
\end{equation*}%

Note that by Lemma \ref{P-invariant highest}, the dimension of
\begin{equation*}
W_{(D_{1},\cdots ,D_{m})}=\left( \rho _{c_{1}}^{D_{1}}\right) ^{{P}%
_{c_{1}}^{\prime }}\otimes \cdots \otimes \left( \rho
_{c_{m}}^{D_{m}}\right) ^{{P}_{c_{m}}^{\prime }}
\end{equation*}%
is at most $1$.
Then the invariant algebra $\mathbb{C}[\mat_{nc}]^{{P}_{\underline{c}%
}^{\prime }}$ is graded by the semigroup ${A}_{\underline{c}}^{+}$ or the set of these $m$%
-tuples $(D_{1},\cdots ,D_{m})$ of Young diagrams, and its graded components are exactly $m$%
-fold tensor products
\begin{equation*}
V_{(D_{1},\cdots ,D_{m})}=\rho _{n}^{D_{1}^{\ast }}\otimes \cdots \otimes
\rho _{n}^{D_{m}^{\ast }}
\end{equation*}%
of irreducible representations of $\GL_{n}$.

$V_{(D_{1},\cdots ,D_{m})}$, as a representation of $\GL_{n}$, can be
decomposed as%
\begin{eqnarray*}
V_{(D_{1},\cdots ,D_{m})} &=&\rho _{n}^{D_{1}^{\ast }}\otimes \cdots \otimes
\rho _{n}^{D_{m}^{\ast }} \\
&=&\sum\limits_{(F_{2},\cdots
,F_{m})}c_{D_{1}D_{2}}^{F_{2}}c_{F_{2}D_{3}}^{F_{3}}\cdots
c_{F_{m-1}D_{m}}^{F_{m}}\left( \rho _{n}^{F^{\ast }}\right)
\end{eqnarray*}%
where $F=F_{m}$ and $c_{F_{i-1}D_{i}}^{F_{i}}$ is the Littlewood-Richardson
number, i.e., the multiplicity of $\rho _{n}^{F_{i}}$ in $\rho
_{n}^{F_{i-1}}\otimes \rho _{n}^{D_{i}}$ with the convention $D_{1}=F_{1}$.
Therefore, $V_{(D_{1},\cdots ,D_{m})}$ contains%
\begin{equation*}
\sum\limits_{(F_{2},\cdots
,F_{m})}c_{D_{1}D_{2}}^{F_{2}}c_{F_{2}D_{3}}^{F_{3}}\cdots
c_{F_{m-1}D_{m}}^{F_{m}}
\end{equation*}%
copies of $\rho _{n}^{F^{*}}$. Note that if $\ell (F_{i})>\min (n,\ell
(F_{i-1})+\ell (D_{i}))$, then $c_{F_{i-1}D_{i}}^{F_{i}}=0$ for all $i$.
Therefore, in particular, $\ell (F)$ should be less than or equal to $\min
(n,c)$.

For $F$ with $\ell (F)\leq \min (n,c)$ and $\dim \left( \rho _{n}^{F}\right)
^{{P}_{n}^{\prime }}=1$, this multiplicity is equal to the dimension of the
invariant space
\begin{eqnarray*}
\left( V_{(D_{1},\cdots ,D_{m})}\right) ^{{P}_{n}^{\prime }}
&=&\sum\limits_{(F_{2},\cdots
,F_{m})}c_{D_{1}D_{2}}^{F_{2}}c_{F_{2}D_{3}}^{F_{3}}\cdots
c_{F_{m-1}D_{m}}^{F_{m}}\left( \rho _{n}^{F^{\ast }}\right) ^{{P}%
_{n}^{\prime }} \\
&\cong &\sum\limits_{F}\Hom_{\GL_{n}}\left( \rho _{n}^{F},\rho
_{n}^{D_{1}}\otimes \cdots \otimes \rho _{n}^{D_{m}}\right) \otimes \left(
\rho _{n}^{F^{\ast }}\right) ^{{P}_{n}^{\prime }}
\end{eqnarray*}%
and we see that the ${P}_{n}^{\prime }$-invariant algebra of $\mathbb{C}[\mat%
_{n,c}]^{{P}_{\underline{c}}^{\prime }}$ has the decomposition%
\begin{eqnarray*}
& &\sum\limits_{(D_{1},\cdots ,D_{m})}\left( V_{(D_{1},\cdots
,D_{m})}\right) ^{{P}_{n}^{\prime }}\otimes W_{(D_{1},\cdots ,D_{m})}
\\
&\cong &\sum\limits_{(D_{1},\cdots ,D_{m})}\sum\limits_{F}\left(
V_{(D_{1},\cdots ,D_{m})}\right) ^{{P}_{n}^{\prime }}   \\
&\cong &\sum\limits_{(D_{1},\cdots ,D_{m})}\sum\limits_{F}\Hom_{\GL%
_{n}}\left( \rho _{n}^{F},\rho _{n}^{D_{1}}\otimes \cdots \otimes \rho
_{n}^{D_{m}}\right) \otimes \left( \rho _{n}^{F^{\ast }}\right) ^{{P}%
_{n}^{\prime }}
\end{eqnarray*}%
and each graded component is an $\left( {A}_{n}^{+}{\times A}_{\underline{c}%
}^{+}\right) $-eigenspace. Consequently, the graded components of the
invariant algebra $\mathbb{C}[\mat_{n,c}]^{{P}_{n}^{\prime }\times {P}_{%
\underline{c}}^{\prime }}$ describe the branching multiplicities under the
restrictions of $\GL_{\underline{n}}$ down to the diagonal $\GL_{n}$.
\end{proof}

\bigskip

\begin{say}
Next, in taking the invariants of $({{P}_{n}^{\prime }\times {P}_{\underline{c}}^{\prime }})$,
by reversing the order of the procedures,
we consider the invariants of ${P}_{n}^{\prime }$ first. This provides
us a different representation theoretic description of the
$({{P}_{n}^{\prime }\times {P}_{\underline{c}}^{\prime }})$-invariant
subalgebra of $\mathbb{C}[\mat_{n,c}]$.
\end{say}

\begin{prop}
The following is an $\left( {A}_{n}^{+}{\times A}_{\underline{c}}^{+}\right)
$-graded algebra decomposition of $\mathbb{C}[\mat_{nk}]^{{P}_{n}^{\prime
}\times {P}_{\underline{c}}^{\prime }}$%
\begin{equation*}
\sum\limits_{F}\sum\limits_{(D_{1},\cdots ,D{m})}\Hom_{\GL_{\underline{c}%
}}\left( \rho _{c_{1}}^{D_{1}}\otimes \cdots \otimes \rho
_{c_{m}}^{D_{m}},\rho _{c}^{F}\right) \otimes \left( \rho
_{c_{1}}^{D_{1}}\right) ^{{P}_{c_{1}}^{\prime }}\otimes \cdots \otimes
\left( \rho _{c_{m}}^{D_{m}}\right) ^{{P}_{c_{m}}^{\prime }}
\end{equation*}%
where the sum runs over $F$ and $D_{i}$ such that $\ell (F)\leq \min (n,c)$,
$\ell (D_{i})\leq \min (n,c_{i})$, and
\begin{equation*}
\dim \left( \rho _{n}^{F}\right) ^{{P}_{n}^{\prime }}=\dim \left( \rho
_{c_{i}}^{D_{i}}\right) ^{{P}_{c_{i}}^{\prime }}=1
\end{equation*}
for $1\leq i\leq m$. Each graded component tells us how a $\GL_{c}$
irreducible representation decomposes as a $\GL_{\underline{c}}$-module.
\end{prop}

\begin{proof}
Starting from the $\GL_{n}$-$\GL_{c}$ duality, we have%
\begin{equation*}
\mathbb{C}[\mat_{n,c}]^{{P}_{n}^{\prime }}=\sum\limits_{\ell (F)\leq \min
(n,c)}(\rho _{n}^{F^{\ast }})^{{P}_{n}^{\prime }}\otimes \rho _{c}^{F}
\end{equation*}

Then, by considering $\rho _{c}^{F}$ as a $\GL_{\underline{c}}$-module,
we have the following decomposition
\begin{eqnarray*}
\rho _{c}^{F} &=&\sum\limits_{(D_{1},\cdots ,D{m})}{m_{(D_{1},\cdots ,D{m}%
)}^{F}}\left( \rho _{c_{1}}^{D_{1}}\otimes \cdots \otimes \rho
_{c_{m}}^{D_{m}}\right)  \\
&\cong &\sum\limits_{(D_{1},\cdots ,D{m})}\Hom_{\GL_{\underline{c}}}\left(
\rho _{c_{1}}^{D_{1}}\otimes \cdots \otimes \rho _{c_{m}}^{D_{m}},\rho
_{c}^{F}\right) \otimes \left( \rho _{c_{1}}^{D_{1}}\otimes \cdots \rho
_{c_{m}}^{D_{m}}\right)
\end{eqnarray*}%
where ${m_{(D_{1},\cdots ,D{m})}^{F}}$ is the multiplicity of the
irreducible $\GL_{\underline{c}}$ module $\rho _{c_{1}}^{D_{1}}\otimes
\cdots \otimes \rho _{c_{m}}^{D_{m}}$ appearing in $\rho _{c}^{F}$.
By further taking its invariants under the action of ${P}_{\underline{c}%
}^{\prime }$, we have the decomposition of $\mathbb{C}[\mat_{n,c}]^{{P}%
_{n}^{\prime }\times {P}_{\underline{c}}^{\prime }}$%
\begin{equation*}
\sum\limits_{F}\sum\limits_{(D_{1},\cdots ,D{m})}\Hom_{\GL_{\underline{c}%
}}\left( \rho _{c_{1}}^{D_{1}}\otimes \cdots \rho _{c_{m}}^{D_{m}},\rho
_{c}^{F}\right) \otimes \left( \rho _{c_{1}}^{D_{1}}\right) ^{{P}%
_{c_{1}}^{\prime }}\otimes \cdots \otimes \left( \rho
_{c_{m}}^{D_{m}}\right) ^{{P}_{c_{m}}^{\prime }}
\end{equation*}%

By Lemma \ref{GLm-GLn}, the dimension of the space $\left( \rho _{c_{1}}^{D_{1}}\right) ^{{P}%
_{c_{1}}^{\prime }}\otimes \cdots \otimes \left( \rho
_{c_{m}}^{D_{m}}\right) ^{{P}_{c_{m}}^{\prime }}$ is at most $1$. Therefore,
each graded component, if it is not zero, encodes the branching rule with respect to the restriction
of $\GL_{c}$ down to $\GL_{\underline{c}}$.
\end{proof}

\begin{say}

We remark that the algebra $\mathbb{C}[\mat_{n,c}]^{{P}_{n}^{\prime }}$ can be understood
as the multi-homogeneous coordinate algebra of the flag variety %
$$Y_{\mathbf{n}}=Fl(n_{1},\cdots ,n_{s};\mathbb{C}^{c})$$%
and its graded component $(\rho_{n}^{F})^{^{{P}_{n}^{\prime }}}\otimes \rho _{c}^{F}$ for $%
F=(f_{1}^{n_{1}},\cdots ,f_{s}^{n_{s}})$ with $f_{i}=e_{i}+\cdots +e_{s}$ is
exactly the section space $\Gamma (Y_{\mathbf{n}},L_{\mathbf{e}})$. See %
\cite[\S 9]{Fu97}. We remark that the graded components are labeled by $F$ and $%
(D_{1},\cdots ,D_{m})$ or equivalently by $\mathbf{e}$ and $\mathbf{r}$. In
fact, it can be identified with the subspace of $\Gamma (Y_{\mathbf{n}},L_{%
\mathbf{e}})$ invariant under ${P}_{\underline{c}}$ and stable under ${P}_{%
\underline{c}}/{P}_{\underline{c}}^{\prime }$ with the character $\mu _{%
\mathbf{r}}$ in (\ref{P-character2}), i.e., $\Gamma \left( Y_{\mathbf{n}},
L_{\mathbf{e}}(\mathbf{r})\right) ^{{P}_{\underline{c}}}$.
\end{say}

\begin{say}
Now, by combining two propositions, we have
\begin{coro}\label{individual reciprocity}
The dimension of the $(\mu _{F},\mu
_{D(1)},\cdots ,\mu _{D(m)})$-homogeneous component for the $\left( {A}%
_{n}^{+}{\times A}_{\underline{c}}^{+}\right) $-graded algebra $\mathbb{C}[%
\mat_{n,c}]^{{P}_{n}^{\prime }\times {P}_{\underline{c}}^{\prime }}$ records
simultaneously

\begin{enumerate}
\item the multiplicity of the $\GL_{n}$ module $\rho _{n}^{F}$ in the tensor
product $\rho _{n}^{D_{1}}\otimes \cdots \otimes \rho _{n}^{D_{m}}$,

\item the multiplicity of the $\GL_{\underline{c}}$ module $\rho
_{c_{1}}^{D_{1}}\otimes \cdots \rho _{c_{m}}^{D_{m}}$ in $\rho _{c}^{F}$.
\end{enumerate}
\end{coro}

\begin{proof}
In the above propositions, we showed that $\mathbb{C}[\mat_{n,c}]^{{P}%
_{n}^{\prime }\times {P}_{\underline{c}}^{\prime }}$, as a $\left( {A}%
_{n}^{+}{\times A}_{\underline{c}}^{+}\right)$-graded algebra, has two
different decompositions%
\begin{eqnarray*}
&&\sum\limits_{(D_{1},\cdots ,D_{m})}\sum\limits_{F}\Hom_{\GL_{n}}\left(
\rho _{n}^{F},\rho _{n}^{D_{1}}\otimes \cdots \otimes \rho
_{n}^{D_{m}}\right) \otimes \left( \rho _{n}^{F^{\ast }}\right) ^{{P}%
_{n}^{\prime }} \\
&&\sum\limits_{F}\sum\limits_{(D_{1},\cdots ,D{m})}\Hom_{\GL_{\underline{c}%
}}\left( \rho _{c_{1}}^{D_{1}}\otimes \cdots \otimes \rho
_{c_{m}}^{D_{m}},\rho _{c}^{F}\right) \otimes \left( \rho
_{c_{1}}^{D_{1}}\right) ^{{P}_{c_{1}}^{\prime }}\otimes \cdots \otimes
\left( \rho _{c_{m}}^{D_{m}}\right) ^{{P}_{c_{m}}^{\prime }}
\end{eqnarray*}%
By comparing the graded components, we see that the dimension
of the following multiplicity spaces should be the same
\begin{eqnarray*}
&&\Hom_{\GL_{n}}\left(
\rho _{n}^{F},\rho _{n}^{D_{1}}\otimes \cdots \otimes \rho
_{n}^{D_{m}}\right); \\
&&\Hom_{\GL_{\underline{c}%
}}\left( \rho _{c_{1}}^{D_{1}}\otimes \cdots \otimes \rho
_{c_{m}}^{D_{m}},\rho _{c}^{F}\right),
\end{eqnarray*}%
which correspond to the branching multiplicities in the statement.
\end{proof}

\end{say}

\begin{rema}
Our proofs for the propositions are the same as
the one given in \cite{HL07} where maximal unipotent subgroups are used
instead of the commutator subgroups of parabolic subgroups.

Corollary \ref{individual reciprocity} shows that the $\left( {P}_{n}^{\prime }\times {P}_{%
\underline{c}}^{\prime }\right) $-invariant subalgebra of $\mathbb{C}[\mat%
_{n,c}]$ encodes two different types of branching rules. With this dual
interpretation, $\mathbb{C}[\mat_{n,c}]^{{P}_{n}^{\prime }\times {P}_{%
\underline{c}}^{\prime }}$ can be called a \textit{reciprocity algebra} in
the sense of \cite{HL07, HTW08}.

Then from (\ref{REP-Ring1}), the ${P}%
_{H}^{\prime }\times {P}_{T}^{\prime }$ invariant subalgebra of $\mathbb{C}[%
\Rep(Q,\bd)]$ can be realized as the tensor product of reciprocity algebras,
and as stated in Theorem \ref{multi-reciprocity}, it encodes two sets of
different types of multiplicity spaces simultaneously. Hence, $%
\mathbb{C}[\Rep(Q,\bd)]^{{P}_{H}^{\prime }\times {P}_{T}^{\prime }}$ can be
considered a \textit{multi-reciprocity algebra}.
\end{rema}

\begin{rema}
It is also possible to develop a parallel theory in terms of
tails starting from
\begin{eqnarray*}
\mathbb{C}[\Rep(Q,\bd)] &=&\bigotimes_{t\in T}\bigotimes_{a\in Q_{1}(t)}%
\mathbb{C}[\mat_{d_{\bh(a)},d_{t}}] \\
&=&\bigotimes_{t\in T}\mathbb{C}[\mat_{n_{t},d_{t}}]
\end{eqnarray*}%
where $Q_{1}(t)=\{a\in Q_{1}:\bt(a)=t\}$ and $n_{t}=\sum_{a\in Q_{1}(t)}d_{\bh(a)}$.
\end{rema}

\begin{say}
We note that there is a nice representation theoretic
interpretation of the geometric condition (\ref{cond}). If the multiplicity
of $\rho _{n}^{F}$ in the tensor product $\rho _{n}^{D}\otimes \rho _{n}^{E}$
is positive, then the number of boxes in $D$ and $E$ is equal to the number
of boxes in $F$. For each $h\in H$, by iterating this condition on Young
diagrams $D(h,i)$ and $F(h)$ in Theorem \ref{multi-reciprocity}, we obtain the
condition: the number of boxes in all $D(h,i)$'s should be equal to the
number of boxes in $F(h)$.

To be more precise, let $F(h)$ and $D(h,i)$ be given as
\begin{eqnarray*}
F(h) &=&((e_{1}+\cdots +e_{s})^{n_{1}},(e_{2}+\cdots +e_{s})^{n_{2}},\cdots
,e_{s}^{n_{s}});  \\
D(h,i) &=&((r_{i,1}+\cdots +r_{i,s_{i}})^{k_{i,1}},(r_{i,2}+\cdots
+r_{i,s_{i}})^{k_{i,2}},\cdots ,r_{i,s_{i}}^{k_{i,s_{i}}})
\end{eqnarray*}%
for each $i$.
Then to ensure that the multiplicity of $\rho _{d_{h}}^{F(h)}$ in the tensor
product%
\begin{equation*}
\rho _{d_{h}}^{D(h,1)}\otimes \cdots \otimes \rho _{d_{h}}^{D(h,m(h))}
\end{equation*}%
of $\GL_{d_{h}}$ representations in Theorem \ref{multi-reciprocity}
to be non-zero, we need%
\begin{equation*}
\sum_{1\leq i\leq m}\sum_{1\leq j\leq s_{i}}k_{i,j}(r_{i,j}+\cdots
+r_{i,s_{i}})=\sum_{1\leq i\leq s}n_{i}(e_{i}+\cdots +e_{s}).
\end{equation*}%
For the whole algebra $\mathbb{C}[\Rep(Q,\bd)]$, by repeating this over all $%
h\in H$ and adding all of them together, we obtain the same condition we
imposed for the linearizations (\ref{cond}).
\end{say}

\begin{say}
As a special case, let us consider the star quiver given in \S \ref{starQuiver}
with the dimension vector $$\bd = (n, 1, \cdots, 1) \in {\mathbb N}^{m+1}$$ %
for $n \leq m$. In this case, the invariant sections can be explicitly described
in terms of the combinatorics of Young tableaux.

\medskip

For a partition $n=n_1 + \cdots + n_s$, let us consider its corresponding parabolic subgroup
$P_n$ of $\GL_n$ and the torus $T_m = (\mathbb{C}^{*})^m$. The section space
$\Gamma (Y_{\bn},L_{\mathbf{e}})$ of $Y_{\bn}=Fl(n_{1},\cdots ,n_{s};\mathbb{C}^{m})$ can
be identified with the summand
$(\rho_{n}^{F^{\ast}})^{P_{n}^{\prime }}\otimes \rho _{m}^{F}\ $of
$\mathbb{C}[\mat_{n,m}]^{P_{n}^{\prime } \times 1}$ for
\begin{eqnarray*}
F&=&(f_1 ^{n_1},f_2 ^{n_2},\cdots,f_s ^{n_s}) \cr
&=&((e_1 + \cdots + e_s)^{n_1}, (e_2 + \cdots +e_s)^{n_2}, \cdots, e_s ^{n_s}).
\end{eqnarray*}
The space $(\rho _{n}^{F^{\ast}})^{P_{n}^{\prime
}}\otimes \rho _{m}^{F}\cong \rho _{m}^{F}$ consists of $f\in \mathbb{C}[%
\mat_{n,m}]$ which are invariant under $P'_n$ and eigenvectors under $P_n/P'_n$
with weight $$\mu_F(\tau_1,\cdots,\tau_s)=\prod \tau_i ^{n_i f_i}.$$ They can be identified
with the products of determinants or semistandard Young tableaux of
diagram $F$ having entries from $\{1,\cdots ,m\}$.

Since the $T_{m}$-eigenspace of $\rho _{m}^{F}$ with weight $\mu_{\mathbf{r}}(t_1,\cdots,t_m)%
=\prod t_i ^{r_i}$ is exactly the space spanned by the weight vectors
of $\rho _{m}^{F}$ with weight $\mu_{\mathbf{r}}$, the elements in the $T_{m}$-invariant section
space $\Gamma (Y_{\bn},L_{\mathbf{e}}(\mathbf{r}))^{T_{m}}$ can be realized as the space
spanned by the products of determinants identifiable with semistandard
Young tableaux of diagram $F$ and content $\mathbf{r}=(r_{1},\cdots ,r_{m})$.
For further detail, see, for example, \cite[\S 9]{Fu97} and \cite{Ki08}.

\end{say}

\bigskip

\section{The Transfer Principle}\label{transfer}

\begin{say}
The so-called \textit{transfer principle} is a useful tool to study
quotients by non-reductive groups.

\begin{theo}\rm{(\cite[Theorem 9.1]{Gro97})}
For a linear algebraic group $G$, let $Z$ be a rational
$G$-module and a subgroup $H$ of $G$ be acting on $\mathbb{C}[G]$ by right
translation. If $Z$ is a $\mathbb{C}$-algebra, then there is an algebra
isomorphism
$$ Z \cong \left( \mathbb{C}[G] \otimes Z \right)^G $$
which is $H$-equivariant. In particular, we have
$$ Z^H \cong \left( \mathbb{C}[G]^H \otimes Z \right)^G.$$
\end{theo}

Let us compare our results with the transfer principle. We recall that
$A^+ _k$ denotes the semigroup of polynomial dominant weights for $\GL_k$.
\begin{prop}
As $(A^+ _n \times A^+ _m)$-graded algebras, we have
$$\mathbb{C}[\mat_{n,m}]^{P'_n \times P'_m} \cong
\left( \mathbb{C}[\GL_m]^{1 \times P'_m} \otimes %
\mathbb{C}[\mat_{n,m}]^{P'_n \times 1}  \right)^{\GL_m} $$
\end{prop}
\begin{proof}
From Lemma \ref{GLm-GLn}, the left hand side decomposes as
$$\sum_{F} (\rho_n ^{F^{*}})^{P'_n} \otimes (\rho_m ^{\lambda})^{P'_m} $$
where the sum runs over all Young diagrams $F$ of length not more than
$min(n,m)$.

For the right hand side, we note that the ring of regular functions
over $GL_m$ decomposes as
$$\mathbb{C}[\GL_m]=\sum_{\lambda}
\rho_m ^{\lambda^{*}} \otimes \rho_m ^{\lambda}$$
over all rational dominant weights $\lambda$ for $\GL_m$ (cf. \cite[Theorem 4.2.7]{GW09}).
By combining this with Lemma \ref{GLm-GLn}, the right hand side
decomposes as
\begin{eqnarray*}
&& \sum_{\lambda, F}
\left(\rho_m ^{\lambda^{*}} \otimes (\rho_m ^{\lambda})^{P'_m} \otimes
(\rho_n ^{F^{*}})^{P'_n} \otimes \rho_m^{F} \right)^{\GL_m} \\
&=& \sum_{\lambda, F}
(\rho_n ^{F^{*}})^{P'_n} \otimes (\rho_m ^{\lambda})^{P'_m}
\otimes \left(\rho_m ^{\lambda^{*}} \otimes \rho_m ^{F} \right)^{\GL_m}
\end{eqnarray*}
Since the dimension of the invariant space $\left(\rho_m ^{\lambda^{*}} \otimes \rho_m ^{F}
\right)^{\GL_m}$ is not more than $1$ and it is exactly $1$
when $\lambda=F$. This shows that two graded algebras are isomorphic.
\end{proof}

For a fence quiver $Q$ with dimension vector $d$, as given in \S \ref{parabolic}
and \S \ref{reciprocity}, by iterating the above proposition, it is easy to see that

\begin{coro}
As $(A_{H}^{+} \times A_{T}^{+})$-graded algebras, we have
$$ \mathbb{C}[Rep(Q,\mathbf{d})]^{P'_H \times P'_T} \cong
\left( \mathbb{C}[G_{T}]^{1 \times {P'_{T}}} \otimes
\mathbb{C}[Rep(Q,\mathbf{d})]^{P'_H \times 1} \right)^{G_T}
$$
where $A_{H}^{+}=\prod_{h}{A^+_{d_h}}$ and $A_{T}^{+}=\prod_{t}{A^+_{d_t}}$.
\end{coro}
\end{say}

\end{document}